\documentclass[11pt]{article}
\usepackage[margin=1in]{geometry}
\usepackage[usenames,dvipsnames,svgnames,table]{xcolor} 

\usepackage{longtable}
\usepackage{hhline}
\usepackage{float}
\usepackage{booktabs}

\usepackage{setspace}
\usepackage{xspace}
\usepackage{enumitem} 
\usepackage{rotfloat} 
\usepackage{graphicx}
\usepackage{amssymb, amsmath, amsthm}
\usepackage{verbatim}
\usepackage{natbib}
\usepackage{authblk}
\usepackage{kotex}
\usepackage{multirow}
\usepackage{subcaption} 
\usepackage{array}
\usepackage{hyperref}
\newcolumntype{L}{>{\centering\arraybackslash}m{0.1\linewidth}}

\newtheorem{theorem}{Theorem}
\newtheorem{lemma}[theorem]{Lemma}
\newtheorem{proposition}[theorem]{Proposition}
\newtheorem{corollary}[theorem]{Corollary}

\newtheorem{assumption}[theorem]{Assumption}

\newcommand{\Frechet}{Fr\'{e}chet\xspace}

\DeclareMathOperator*{\argmin}{arg\,min}

\DeclareMathOperator{\Var}{Var}
\DeclareMathOperator{\Cov}{Cov}
\DeclareMathOperator{\inj}{inj}
\DeclareMathOperator{\secop}{sec}

\DeclareMathOperator{\Cut}{Cut}

\newcommand{\M}{\mathcal M}
\newcommand{\Pbb}{\mathbb P}

\newcommand{\R}{\mathbb R}

\newcommand{\eps}{\varepsilon}
\newcommand{\1}{\mathbf{1}}

\begin{document}
	
\title{Scale selection for geometric medians on product manifolds}
\author[1,2]{Kisung You}
\affil[1]{Department of Mathematics, Baruch College}
\affil[2]{Department of Mathematics, The Graduate Center, City University of New York}
\date{}

\maketitle
\begin{abstract}
Geometric medians on product manifolds are sensitive to the relative scaling of factor metrics because the median objective couples the factors rather than separating them.  We study this scale-selection problem and first prove that naive joint minimization over location and scale is degenerate: the scale is driven to the boundary and the problem collapses to a marginal median, effectively discarding one factor.  Thus relative scale is not identifiable from the raw median loss alone. We develop three alternatives to mitigate this issue. The first treats scale as indexing a sensitivity path and establishes uniform consistency, a functional central limit theorem, and a derivative-based sensitivity measure.  The second constructs a robust scale-calibrated median using marginal radial median scales, yielding unit invariance, consistency, a two-step central limit theorem, and bounded influence.  The third introduces a bounded balance equation for direct scale estimation, with monotonicity, uniqueness, joint asymptotic normality, and bounded influence.  Simulations illustrate boundary collapse, sensitivity, unit invariance, and balanced estimation in Euclidean and Bures--Wasserstein settings.
\end{abstract}


\section{Introduction}

Robust location estimation on nonlinear spaces has become a central problem in
modern statistics.  Robustness has long been a foundational theme in statistical
inference \citep{huber_1981_RobustStatistics, hampel_2005_RobustStatisticsApproach}, while manifold-valued and object-valued data
analysis has developed into a broad area of modern multivariate statistics \citep{bhattacharya_2012_NonparametricInferenceManifolds, patrangenaru_2016_NonparametricStatisticsManifolds, pennec_2020_RiemannianGeometricStatistics}.  Notable examples include
Riemannian manifolds, quotient spaces, and metric spaces that appear in
directional data analysis, shape analysis, covariance matrix analysis, optimal
transport, robotics, medical imaging, and multimodal data integration.  The
\Frechet mean is the most widely used notion of center on such spaces \citep{frechet_1948_ElementsAleatoiresNature, bhattacharya_2003_LargeSampleTheory}, but it inherits the
well-known nonrobustness of Euclidean means.  The geometric median, defined as
the minimizer of expected distance rather than expected squared distance,
provides a natural robust alternative \citep{vardi_2000_MultivariateL1medianAssociated, minsker_2015_GeometricMedianRobust, fletcher_2009_GeometricMedianRiemannian, you_2025_WassersteinMedianProbability}.

This paper concerns geometric medians on product manifolds.  Let $(M,g_M)$ and $(N,g_N)$ be complete connected Riemannian manifolds with distances $d_M$ and $d_N$.  For paired observations
$
 Z=(X,Y)\in \M:=M\times N,
$
the canonical product distance is
\[
 d_1((p,q),(x,y))=\bigl\{d_M(p,x)^2+d_N(q,y)^2\bigr\}^{1/2}.
\]
The product geometric median is a minimizer of
\[
 m\mapsto E[d_1(m,Z)],\qquad m=(p,q)\in M\times N.
\]
A key feature, emphasized in the fixed-metric theory of \citet{park_2026_GeometricMediansProduct}, is that this objective is not separable across factors.  Indeed,
\[
 E[d_1((p,q),(X,Y))]
 =E\left[\{d_M(p,X)^2+d_N(q,Y)^2\}^{1/2}\right]
\]
contains a common nonlinear norm that couples the two components. This differs sharply from the Fr\'echet mean objective,
\[
 E[d_1((p,q),(X,Y))^2]
 =E[d_M(p,X)^2]+E[d_N(q,Y)^2],
\]
which separates additively.

The nonseparability of the product median has an important consequence that the relative scaling of the factor metrics changes the median target.  To make this explicit, we introduce the scaled product metric
\begin{equation}\label{eq:scaled-distance-intro}
 d_\alpha((p,q),(x,y))
 =\left\{\alpha d_M(p,x)^2+(2-\alpha)d_N(q,y)^2\right\}^{1/2},
 \qquad \alpha\in(0,2).
\end{equation}
The normalization is chosen so that $\alpha=1$ corresponds to the canonical product metric.  Relative to the canonical metric, $\alpha>1$ increases the weight of the
$M$-factor, while $\alpha<1$ increases the weight of the $N$-factor. For each fixed $\alpha$, define the population and empirical median objectives
\[
 F_\alpha(m)=E[d_\alpha(m,Z)],
 \qquad
 F_{n,\alpha}(m)=\frac1n\sum_{i=1}^n d_\alpha(m,Z_i),
\]
and the corresponding medians
\[
 m_\alpha=\argmin_{m\in\M}F_\alpha(m),
 \qquad
 \widehat m_{n,\alpha}=\argmin_{m\in\M}F_{n,\alpha}(m).
\]

The motivating question is as follows: can the relative scale parameter $\alpha$ be learned jointly with the product geometric median by minimizing the same median objective over both $m$ and $\alpha$? 

Our answer is negative. For every fixed location $m$, the map
$
 \alpha\mapsto F_{n,\alpha}(m)
$
is concave.  Hence an interior critical point in $\alpha$ is generically a maximum, not a minimum.  Minimization over $\alpha$ is likely to drive the solution to the boundary.  At the population level, minimizing over $\alpha\in(0,2)$  collapses to the factor with the smaller marginal objective. Thus $\alpha$ is not identifiable from the raw median loss alone.

This negative result is the starting point of the paper.  It clarifies that scale choice for product medians is not a harmless tuning detail.  It is a target-defining decision and any meaningful selection or estimation of $\alpha$ requires an additional principle beyond minimizing $F_\alpha(m)$ jointly in $(m,\alpha)$. Our main contributions are as follows.

\begin{enumerate}
\item \textbf{Failure of naive joint learning.}
We prove that minimizing the raw median objective jointly over location and
scale is degenerate: the scale parameter is driven to the boundary, so the
procedure effectively favors one factor and discards the other.

\item \textbf{Sensitivity of the median path.}
We treat the scale parameter as indexing a family of population medians and
study the resulting median path.  We prove uniform consistency, a functional
central limit theorem, and a derivative formula that quantifies scale
sensitivity.

\item \textbf{Robust scale calibration.}
We propose a default scale choice based on robust marginal radial scales.  The
resulting median is equivariant to changes of units in the factor metrics, and
we establish consistency, a two-step central limit theorem, and bounded
influence.

\item \textbf{Balanced scale estimation.}
We introduce a bounded estimating equation that selects the scale by balancing
the relative factor contributions.  This yields a direct data-adaptive scale
estimator, for which we prove uniqueness, consistency, joint asymptotic
normality, and bounded influence.
\end{enumerate}

The rest of the paper is organized as follows.  Section~\ref{sec:setup}
introduces the scaled product metric, notation, and standing assumptions.
Section~\ref{sec:nogo} proves the no-go theorem for naive joint scale learning.
Section~\ref{sec:path} develops the sensitivity-path theory.  Section~\ref{sec:scale-calibration}
studies robust scale calibration, and Section~\ref{sec:balanced} introduces the
balanced estimating equation.  Section~\ref{sec:experiments} presents numerical
experiments.  Section~\ref{sec:conclusion} concludes with a discussion of the
main implications and future directions. Proofs of the formal results are collected in the appendix.

\section{Setup and scaled product geometry}\label{sec:setup}

\subsection{Scaled product metrics}

Let $(M,g_M)$ and $(N,g_N)$ be smooth, connected, complete Riemannian manifolds.  Let $d_M$ and $d_N$ denote their geodesic distances.  Denote by 
$
 \M=M\times N
$
the product manifold.

For $m=(p,q)\in\M$ and $z=(x,y)\in\M$, we write the factor distances as
$d_M(p,x)$ and $d_N(q,y)$. For $\alpha\in(0,2)$, define the scaled product metric tensor
\begin{equation}\label{eq:g-alpha}
 g_\alpha=\alpha g_M\oplus(2-\alpha)g_N.
\end{equation}
The associated geodesic distance is
\begin{equation}\label{eq:d-alpha}
 d_\alpha((p,q),(x,y))
 =\left\{\alpha d_M(p,x)^2+(2-\alpha)d_N(q,y)^2\right\}^{1/2}.
\end{equation}

Scaling a Riemannian metric by a positive constant leaves the Levi--Civita
connection unchanged.  Hence the factorwise geodesics and logarithmic maps are
the same as for the unscaled metrics, while their lengths are rescaled \citep{you_2026_ConstantMetricScaling}.  Thus
\[
 \log_{(p,q)}(x,y)=(\log_p x,\log_q y)
\]
under all metrics $g_\alpha$, although its norm depends on $\alpha$:
\[
 \|\log_{(p,q)}(x,y)\|_{g_\alpha}^2
 =\alpha\|\log_p x\|_{g_M}^2+(2-\alpha)\|\log_q y\|_{g_N}^2.
\]
Throughout we use the compact scale interval
\begin{equation}\label{eq:Ieps}
 I_\eps=[\eps,2-\eps],
 \qquad 0<\eps<1,
\end{equation}
for uniform results.  This excludes geometric degeneration as one factor weight tends to zero.

\subsection{Curvature and injectivity under scaling}

The following lemma records basic quantities such as sectional curvature and injectivity radius for the scaled metric.

\begin{lemma}[Scaled product geometry]\label{lem:scaled-geometry}
Assume $\secop_M\le \kappa_M$ and $\secop_N\le \kappa_N$.  Under the metric $g_\alpha=\alpha g_M\oplus(2-\alpha)g_N$, the sectional curvature of the product satisfies
\[
 \secop_{g_\alpha}\le
 \kappa_\alpha
 :=
 \max\left\{
 \frac{\kappa_M}{\alpha},
 \frac{\kappa_N}{2-\alpha}
 \right\}.
\]
Moreover,
\[
 \inj_{\alpha g_M}(p)=\sqrt\alpha\,\inj_{g_M}(p),
 \qquad
 \inj_{(2-\alpha)g_N}(q)=\sqrt{2-\alpha}\,\inj_{g_N}(q).
\]
Consequently, on $I_\eps$,
\[
 \kappa_\alpha\le
 \kappa_\eps:=\max\left\{\frac{\kappa_M}{\eps},\frac{\kappa_N}{\eps}\right\},
\]
and the scaled injectivity radii are bounded below by $\sqrt\eps$ times their unscaled counterparts.
\end{lemma}

\subsection{Population and empirical medians}

For completeness, we revisit how geometric medians on product manifolds are defined. 
Let $Z=(X,Y)$ be a random element of $\M$ with distribution $P$, and let $Z_1,\ldots,Z_n$ be i.i.d. copies of $Z$.  For $m=(p,q)$,  define
\begin{equation}\label{eq:Falpha}
 F_\alpha(m)=E[d_\alpha(m,Z)],
 \qquad
 F_{n,\alpha}(m)=P_n d_\alpha(m,\cdot)=\frac1n\sum_{i=1}^n d_\alpha(m,Z_i).
\end{equation}
A population scaled product median is any minimizer
\[
 m_\alpha\in\argmin_{m\in\M}F_\alpha(m),
\]
and an empirical scaled product median is any minimizer
\[
 \widehat m_{n,\alpha}\in\argmin_{m\in\M}F_{n,\alpha}(m).
\]
When uniqueness holds, the set-valued notation is suppressed.

For $m=(p,q)$, write
\[
 u_m(Z)=\log_pX\in T_pM,
 \qquad
 v_m(Z)=\log_qY\in T_qN,
\]
and
\begin{equation}\label{eq:r-alpha}
 r_\alpha(m,Z)=d_\alpha(m,Z)
 =\left\{\alpha\|u_m(Z)\|_M^2+(2-\alpha)\|v_m(Z)\|_N^2\right\}^{1/2}.
\end{equation}
On the set where $Z\notin\Cut(m)\cup\{m\}$, define the median score
\begin{equation}\label{eq:score}
 \psi_\alpha(Z;m)
 =\frac{(u_m(Z),v_m(Z))}{r_\alpha(m,Z)}
 \in T_m\M.
\end{equation}
The population score is $ G_\alpha(m)=E[\psi_\alpha(Z;m)]$ and the first-order  condition for an interior nonsingular population median is $G_\alpha(m_\alpha)=0.$
The Riemannian gradient of $m\mapsto d_\alpha(m,Z)$ with respect to the scaled metric $g_\alpha$ is $-\psi_\alpha(Z;m)$.  With respect to the unscaled product metric $g_1$, the gradient is
\[
 -\frac{(\alpha u_m(Z),(2-\alpha)v_m(Z))}{r_\alpha(m,Z)}.
\]
The zero equations are equivalent because $\alpha$ and $2-\alpha$ are positive.  We use \eqref{eq:score} because it gives the cleanest common notation over $\alpha$.

\subsection{Standing assumptions}

The following assumptions are intentionally explicit.  They are local versions of standard conditions for nonsmooth Riemannian $M$-estimation and hold, for example, globally on Hadamard products under non-collinearity and locally on sufficiently small strongly convex balls in bounded positive curvature. 

\begin{assumption}[Uniform local geometry]\label{ass:geometry}
There exists a compact geodesically strongly convex set $K\subset\M$ such that, for every $\alpha\in I_\eps$, the population median $m_\alpha$ belongs to the interior of $K$.  On an open neighborhood $K^+$ of $K$, minimizing geodesics are unique, logarithmic maps are single-valued and smooth away from the diagonal, and all relevant curvature and injectivity constants are bounded uniformly over $\alpha\in I_\eps$.
\end{assumption}

\begin{assumption}[Unique medians and uniform separation]\label{ass:separation}
For each $\alpha\in I_\eps$, $F_\alpha$ has a unique minimizer $m_\alpha$ in $K$.  Moreover, for every $\eta>0$,
\begin{equation}\label{eq:uniform-separation}
 \inf_{\alpha\in I_\eps}\ 
 \inf_{m\in K:\ d_1(m,m_\alpha)\ge\eta}
 \{F_\alpha(m)-F_\alpha(m_\alpha)\}>0.
\end{equation}
\end{assumption}

\begin{assumption}[No singular mass]\label{ass:no-singular}
For every $\alpha\in I_\eps$,
\[
 P(Z=m_\alpha)=0,
\]
and for every $m$ in a neighborhood of $\{m_\alpha:\alpha\in I_\eps\}$,
\[
 P(Z\in\Cut(m))=0.
\]
\end{assumption}

\begin{assumption}[Differentiability and nonsingularity]\label{ass:differentiability}
The map $(\alpha,m)\mapsto G_\alpha(m)$ is continuously differentiable in a neighborhood of
$
 \{(\alpha,m_\alpha):\alpha\in I_\eps\}.
$
For each $\alpha$, let
\[
 A_\alpha=D_mG_\alpha(m_\alpha):T_{m_\alpha}\M\to T_{m_\alpha}\M.
\]
Then $A_\alpha$ is invertible and
\[
 \inf_{\alpha\in I_\eps}\sigma_{\min}(A_\alpha)>0
\]
when represented in any smooth orthonormal frame over the compact set of medians.
\end{assumption}

\begin{assumption}[Uniform empirical-process regularity]\label{ass:empirical}
The classes
\[
 \mathcal F_K=\bigl\{z\mapsto d_\alpha(m,z):\alpha\in I_\eps,\ m\in K\bigr\}
\]
and, in local charts and after parallel transport to a common tangent space,
\[
 \mathcal G_K=\bigl\{z\mapsto \psi_\alpha(z;m):\alpha\in I_\eps,\ m\in K\bigr\}
\]
are respectively $P$-Glivenko--Cantelli and $P$-Donsker.  In addition, the stochastic equicontinuity implied by the Donsker condition holds uniformly in $\alpha$ and $m$ in neighborhoods of $m_\alpha$.
\end{assumption}

These assumptions are standard in empirical-process treatments of
$M$- and $Z$-estimators \citep{vandervaart_1998_AsymptoticStatistics, kosorok_2008_IntroductionEmpiricalProcesses}. We note that Assumption \ref{ass:empirical} can be replaced by primitive smoothness and entropy conditions.  On compact normal-coordinate neighborhoods, the functions above are Lipschitz-parametric away from a $P$-null singular set, with bounded envelopes on $I_\eps$.  Standard results for Lipschitz-parametric empirical-process classes then imply the stated Glivenko--Cantelli and Donsker properties.

\section{Why naive joint learning of $\alpha$ fails}\label{sec:nogo}

This section proves that $\alpha$ cannot be learned by minimizing the raw median
objective jointly over $(m,\alpha)$.  The obstruction comes from a structural
feature of the geometric median loss. Then the $\alpha$-scaled distance can be written as
\[
 d_\alpha(m,z)=\{\alpha A_m(z)+(2-\alpha)B_m(z)\}^{1/2}.
\]

\begin{lemma}[Concavity in the scale parameter]\label{lem:concavity}
For every fixed $m\in\M$ and $z\in\M$, the map
\[
 \alpha\mapsto d_\alpha(m,z)
\]
is concave on $(0,2)$.  Wherever $d_\alpha(m,z)>0$,
\[
 \frac{\partial}{\partial\alpha}d_\alpha(m,z)
 =\frac{A_m(z)-B_m(z)}{2d_\alpha(m,z)},
\]
and
\[
 \frac{\partial^2}{\partial\alpha^2}d_\alpha(m,z)
 =-\frac{\{A_m(z)-B_m(z)\}^2}{4d_\alpha(m,z)^3}\le0.
\]
Consequently, $\alpha\mapsto F_{n,\alpha}(m)$ and $\alpha\mapsto F_\alpha(m)$ are concave wherever the derivatives are well-defined; the concavity statement holds globally by continuity.
\end{lemma}

\begin{theorem}[Boundary degeneracy of naive joint learning]\label{thm:boundary-degeneracy}
Let $[a,b]\subset(0,2)$ be a compact interval.  Then
\begin{equation}\label{eq:emp-boundary-degeneracy}
 \inf_{m\in\M,\ \alpha\in[a,b]}F_{n,\alpha}(m)
 =
 \min\left\{
 \inf_{m\in\M}F_{n,a}(m),
 \inf_{m\in\M}F_{n,b}(m)
 \right\}.
\end{equation}
The same identity holds at the population level:
\begin{equation}\label{eq:pop-boundary-degeneracy}
 \inf_{m\in\M,\ \alpha\in[a,b]}F_\alpha(m)
 =
 \min\left\{
 \inf_{m\in\M}F_a(m),
 \inf_{m\in\M}F_b(m)
 \right\}.
\end{equation}
\end{theorem}

\begin{corollary}[Endpoint collapse on the open interval]\label{cor:endpoint-collapse}
Assume $E[d_M(p_0,X)]<\infty$ and $E[d_N(q_0,Y)]<\infty$ for some $p_0\in M$ and $q_0\in N$.  Then
\begin{equation}\label{eq:endpoint-collapse}
 \inf_{m\in\M,\ \alpha\in(0,2)}F_\alpha(m)
 =
 \sqrt2\min\left\{
 \inf_{p\in M}E[d_M(p,X)],
 \inf_{q\in N}E[d_N(q,Y)]
 \right\}.
\end{equation}
\end{corollary}

One may consider an alternative restriction  $\alpha\in[\eps,2-\eps]$.  Unfortunately, this does not solve the identification problem.  The naive joint minimizer then selects one of the artificial endpoints $\eps$ or $2-\eps$.  Thus the issue is not merely that $\alpha=0$ and $\alpha=2$ are inadmissible.  The raw objective has the wrong curvature for scale selection. This observation extends to the product of more than two manifolds. 

\begin{corollary}[Finite products]\label{cor:kfactor}
Let $\M=M_1\times\cdots\times M_K$ and let
\[
 d_w(m,z)^2=\sum_{j=1}^K w_jd_j(m_j,z_j)^2,
 \qquad
 w\in\Delta_K:=\{w_j\ge0,\ \sum_{j=1}^Kw_j=1\}.
\]
For fixed $m$ and $z$, the map $w\mapsto d_w(m,z)$ is concave on $\Delta_K$.  Hence naive minimization over $w$ of the product geometric-median objective attains its infimum at an extreme point of the simplex, thereby selecting a single factor.
\end{corollary}

Theorems \ref{thm:boundary-degeneracy} and Corollary \ref{cor:endpoint-collapse} establish the central point that $\alpha$ is not identifiable from the raw product-median objective alone.  The remainder of the paper develops three alternatives, each corresponding to a distinct  principle.

\section{Approach I: the sensitivity path}\label{sec:path}

The no-go theorem shows that $\alpha$ cannot be learned by minimizing the raw
median objective jointly over location and scale.  A first alternative is to
avoid selecting a single scale and instead study the whole path
$\alpha\mapsto m_\alpha$.  This turns scale dependence into an inferential
object: one can estimate how the product median moves as the relative factor
weight changes.

\subsection{Continuity of the population path}

We first record that the population path is well behaved under the same
compactness and uniqueness assumptions used for fixed-scale medians.

\begin{proposition}[Continuity of the median path]\label{prop:path-continuity}
Under Assumptions \ref{ass:geometry} and \ref{ass:separation}, the map
\[
 I_\eps\ni\alpha\mapsto m_\alpha\in K
\]
is continuous.
\end{proposition}

Thus the fixed-scale medians form a genuine path, rather than an unrelated
collection of minimizers.

\subsection{Uniform consistency}

The next result strengthens pointwise consistency to uniform convergence over
the scale interval.

\begin{theorem}[Uniform consistency of the empirical path]\label{thm:uniform-consistency}
Under Assumptions \ref{ass:geometry}, \ref{ass:separation}, and \ref{ass:empirical}, any measurable selection $\widehat m_{n,\alpha}$ satisfying
\[
 F_{n,\alpha}(\widehat m_{n,\alpha})
 \le
 \inf_{m\in K}F_{n,\alpha}(m)+o(1)
\]
uniformly in $\alpha\in I_\eps$ obeys
\begin{equation}\label{eq:uniform-consistency}
 \sup_{\alpha\in I_\eps}d_1(\widehat m_{n,\alpha},m_\alpha)\to0
\end{equation}
in outer probability.  If the Glivenko--Cantelli convergence in Assumption \ref{ass:empirical} holds almost surely, then \eqref{eq:uniform-consistency} holds almost surely.
\end{theorem}

The restriction to $I_\eps$ keeps the product geometry uniformly nondegenerate
and is used throughout the pathwise theory.

\subsection{Functional central limit theorem}

For pathwise inference, tangent vectors based at different $m_\alpha$ must be
compared in one vector space.  We therefore parallel transport all fluctuations
to a fixed reference tangent space.  Fix $\alpha_0\in I_\eps$.  By Assumption
\ref{ass:geometry} and Proposition \ref{prop:path-continuity}, after possibly
shrinking $K^+$, there are unique minimizing geodesics from $m_\alpha$ to
$m_{\alpha_0}$ for all $\alpha\in I_\eps$. Let
\[
 \Pi_\alpha:T_{m_\alpha}\M\to T_{m_{\alpha_0}}\M
\]
denote parallel transport along these geodesics.  Define
\[
 \mathbb Z_n(\alpha)
 =\frac1{\sqrt n}\sum_{i=1}^n
 \Pi_\alpha\psi_\alpha(Z_i;m_\alpha)
 \in T_{m_{\alpha_0}}\M.
\]
By the Donsker assumption, $\mathbb Z_n$ converges weakly in $\ell^\infty(I_\eps,T_{m_{\alpha_0}}\M)$ to a tight mean-zero Gaussian process $\mathbb Z$ with covariance kernel
\begin{equation}\label{eq:Z-cov-kernel}
 \Cov(\mathbb Z(\alpha),\mathbb Z(\beta))
 =
 \Cov\{\Pi_\alpha\psi_\alpha(Z;m_\alpha),
       \Pi_\beta\psi_\beta(Z;m_\beta)\}.
\end{equation}
Let
\[
 \widetilde A_\alpha= \Pi_\alpha A_\alpha\Pi_\alpha^{-1}
 :T_{m_{\alpha_0}}\M\to T_{m_{\alpha_0}}\M.
\]

\begin{theorem}[Functional CLT for the median path]\label{thm:functional-clt}
Assume \ref{ass:geometry}--\ref{ass:empirical}.  Suppose the empirical medians solve the local estimating equations
\[
 G_{n,\alpha}(\widehat m_{n,\alpha})=0
\]
with probability tending to one, uniformly over $\alpha\in I_\eps$, and that the usual uniform linearization of $G_{n,\alpha}$ around $m_\alpha$ holds.  Then
\begin{equation}\label{eq:uniform-linear-expansion}
 \sup_{\alpha\in I_\eps}
 \left\|
 \sqrt n\,\Pi_\alpha\log_{m_\alpha}(\widehat m_{n,\alpha})
 +
 \widetilde A_\alpha^{-1}\mathbb Z_n(\alpha)
 \right\|
 \to0
\end{equation}
in probability.  Consequently,
\begin{equation}\label{eq:functional-clt}
 \left\{
 \sqrt n\,\Pi_\alpha\log_{m_\alpha}(\widehat m_{n,\alpha})
 :\alpha\in I_\eps
 \right\}
 \rightsquigarrow
 \left\{-\widetilde A_\alpha^{-1}\mathbb Z(\alpha):\alpha\in I_\eps\right\}
\end{equation}
in $\ell^\infty(I_\eps,T_{m_{\alpha_0}}\M)$.
\end{theorem}

Pointwise normality is recovered by evaluating the functional limit at a fixed
scale value. For each fixed $\alpha$, Theorem \ref{thm:functional-clt} gives
\[
 \sqrt n\log_{m_\alpha}(\widehat m_{n,\alpha})
 \Rightarrow
 N(0,V_\alpha),
\]
where
\begin{equation}\label{eq:fixed-alpha-cov}
 V_\alpha=A_\alpha^{-1}\Sigma_\alpha A_\alpha^{-T},
 \qquad
 \Sigma_\alpha=\Var\{\psi_\alpha(Z;m_\alpha)\}.
\end{equation}
The full covariance kernel of the path is obtained from \eqref{eq:Z-cov-kernel} by applying $\widetilde A_\alpha^{-1}$ and $\widetilde A_\beta^{-1}$. The functional version additionally captures joint fluctuations across
different scale values and supports simultaneous confidence bands for scalar
summaries of the path.

\subsection{Differentiability of the population path}

We now quantify the deterministic movement of the population target as
$\alpha$ changes.  Differentiating the score equation by the implicit function
theorem gives the following sensitivity formula.

\begin{theorem}[Derivative of the median path]\label{thm:path-derivative}
Assume \ref{ass:geometry}, \ref{ass:no-singular}, and \ref{ass:differentiability}.  Then the map $\alpha\mapsto m_\alpha$ is continuously differentiable on $I_\eps$.  Its derivative $\dot m_\alpha\in T_{m_\alpha}\M$ satisfies
\begin{equation}\label{eq:path-derivative}
 \dot m_\alpha=-A_\alpha^{-1}B_\alpha,
\end{equation}
where
\[
 B_\alpha=\partial_\alpha G_\alpha(m)\big|_{m=m_\alpha}\in T_{m_\alpha}\M.
\]
In the score convention \eqref{eq:score},
\begin{equation}\label{eq:Balpha}
 B_\alpha
 =
 E\left[
 -\frac{d_M(p_\alpha,X)^2-d_N(q_\alpha,Y)^2}
 {2d_\alpha(m_\alpha,Z)^3}
 \log_{m_\alpha}Z
 \right],
\end{equation}
where $m_\alpha=(p_\alpha,q_\alpha)$ and $\log_{m_\alpha}Z=(\log_{p_\alpha}X,\log_{q_\alpha}Y)$.
\end{theorem}

The formula separates scale sensitivity into two parts: $B_\alpha$ measures how
the score changes with scale, while $A_\alpha^{-1}$ measures the local
conditioning of the median equation.

\subsection{Sensitivity indices and confidence bands}
The derivative suggests simple scalar diagnostics.  The intrinsic sensitivity
index is
\begin{equation}\label{eq:sensitivity-index}
 S(\alpha)=\|\dot m_\alpha\|_{g_\alpha}
 =\|A_\alpha^{-1}B_\alpha\|_{g_\alpha}.
\end{equation}
Factorwise sensitivities can also be reported:
\[
 S_M(\alpha)=\|\dot p_\alpha\|_{g_M},
 \qquad
 S_N(\alpha)=\|\dot q_\alpha\|_{g_N}.
\]
Large values indicate local instability with respect to relative metric scaling.

The functional CLT also yields simultaneous confidence bands for scalar summaries of the path.  For example, if $\varphi:\M\to\R$ is smooth, the delta method gives
\[
 \sqrt n\{\varphi(\widehat m_{n,\alpha})-\varphi(m_\alpha)\}
 \rightsquigarrow
 D\varphi_{m_\alpha}\{-A_\alpha^{-1}\mathbb Z(\alpha)\}
\]
in $\ell^\infty(I_\eps)$.  Bootstrap versions are obtained by resampling observations and recomputing the whole path on a grid. A rigorous bootstrap theorem follows from the same arguments as Theorem~\ref{thm:functional-clt}, replacing $P$ by $P_n$ and using conditional Donsker convergence \citep{bickel_1981_AsymptoticTheoryBootstrap, vandervaart_1996_WeakConvergenceEmpirical, bhattacharya_2016_OmnibusCLTsFrechet}.

\section{Approach II: robust scale calibration}\label{sec:scale-calibration}

The sensitivity path is useful when one wants to report scale dependence.  In
many applications, however, a single default median is needed.  We propose to
obtain it by standardizing each factor distance with a robust marginal scale
before forming the product metric.

\subsection{Marginal medians and marginal radial scales}

The calibration uses marginal information only.  This separates unit
standardization from the coupled product-median objective and avoids the
degeneracy of joint minimization over $\alpha$.  Let $p_0$ and $q_0$ be the
marginal geometric medians:
\begin{equation}\label{eq:marginal-medians}
 p_0=\argmin_{p\in M}E[d_M(p,X)],
 \qquad
 q_0=\argmin_{q\in N}E[d_N(q,Y)].
\end{equation}
Define marginal radial random variables
\[
 R_M=d_M(p_0,X),
 \qquad
 R_N=d_N(q_0,Y).
\]
Let
\begin{equation}\label{eq:pop-scales}
 s_M=\operatorname{med}(R_M),
 \qquad
 s_N=\operatorname{med}(R_N),
\end{equation}
where $\operatorname{med}$ denotes the lower or unique median.  We assume throughout this section that
\[
 0<s_M<\infty,
 \qquad
 0<s_N<\infty.
\]

The empirical marginal medians are
\begin{equation}\label{eq:emp-marginal-medians}
 \widehat p_n\in\argmin_{p\in M}\frac1n\sum_{i=1}^n d_M(p,X_i),
 \qquad
 \widehat q_n\in\argmin_{q\in N}\frac1n\sum_{i=1}^n d_N(q,Y_i),
\end{equation}
and the empirical radial scales are
\begin{equation}\label{eq:emp-scales}
 \widehat s_M=\operatorname{med}_{1\le i\le n}\ d_M(\widehat p_n,X_i),
 \qquad
 \widehat s_N=\operatorname{med}_{1\le i\le n}\ d_N(\widehat q_n,Y_i).
\end{equation}
The scale-calibrated weight is
\begin{equation}\label{eq:alpha-sc}
 \widehat\alpha_{\rm sc}
 =2\frac{\widehat s_M^{-2}}{\widehat s_M^{-2}+\widehat s_N^{-2}}
 =2\frac{\widehat s_N^2}{\widehat s_M^2+\widehat s_N^2}.
\end{equation}
The corresponding calibrated product median is
\begin{equation}\label{eq:calibrated-median}
 \widehat m_{\rm sc}=\widehat m_{n,\widehat\alpha_{\rm sc}}.
\end{equation}
If one factor has larger marginal spread, its inverse-scale weight is smaller;
hence calibration prevents that factor from dominating only because of units.

\subsection{Equivalence to standardized distances}

The formula for $\widehat\alpha_{\rm sc}$ is a normalized version of
standardized product-distance minimization.  The scale-calibrated procedure is
equivalent to minimizing
\begin{equation}\label{eq:standardized-distance}
 d_{\rm std}(m,Z)^2
 =\frac{d_M(p,X)^2}{\widehat s_M^2}
 +\frac{d_N(q,Y)^2}{\widehat s_N^2}.
\end{equation}
Indeed, setting
\[
 \lambda_M=\widehat s_M^{-2},
 \qquad
 \lambda_N=\widehat s_N^{-2},
\]
and normalizing $\lambda_M+\lambda_N=2$ gives \eqref{eq:alpha-sc}.  

Multiplying the entire distance by a positive constant does not change the
median. The key property of this construction is exact invariance to multiplicative
changes of units in either factor metric.
\begin{theorem}[Scale equivariance]\label{thm:scale-equivariance}
Suppose the factor distances are rescaled by positive constants:
\[
 d_M' = c_Md_M,
 \qquad
 d_N'=c_Nd_N.
\]
Let $\widehat m_{\rm sc}'$ be the calibrated product median computed from $d_M'$ and $d_N'$.  Then
\[
 \widehat m_{\rm sc}'=\widehat m_{\rm sc},
\]
provided the same measurable selections are used when minimizers are nonunique.
\end{theorem}
Thus the calibrated median depends on the relative geometry of the data, not on
the arbitrary units used to measure the factor distances.

\subsection{Consistency}

The population target of calibration is obtained by replacing the empirical
marginal scales with their population limits. 

\begin{assumption}[Marginal scale consistency]\label{ass:scale-consistency}
The marginal medians and scales satisfy
\[
 \widehat p_n\to p_0,
 \qquad
 \widehat q_n\to q_0,
 \qquad
 \widehat s_M\to s_M,
 \qquad
 \widehat s_N\to s_N
\]
in probability, with $s_M,s_N>0$.
\end{assumption}

\begin{theorem}[Consistency of the calibrated scale and median]\label{thm:calibration-consistency}
Under Assumption \ref{ass:scale-consistency},
\begin{equation}\label{eq:alpha-sc-limit}
 \widehat\alpha_{\rm sc}\to
 \alpha_{\rm sc}:=2\frac{s_M^{-2}}{s_M^{-2}+s_N^{-2}}
 =2\frac{s_N^2}{s_M^2+s_N^2}
\end{equation}
in probability.  If, in addition, Theorem \ref{thm:uniform-consistency} holds on an interval $I_\eps$ containing $\alpha_{\rm sc}$ and $P(\widehat\alpha_{\rm sc}\in I_\eps)\to1$, then
\[
 \widehat m_{n,\widehat\alpha_{\rm sc}}\to m_{\alpha_{\rm sc}}
\]
in probability.
\end{theorem}
Therefore the calibrated estimator targets the scale-standardized median
$m_{\alpha_{\rm sc}}$.

\subsection{Two-step asymptotic distribution}

We now account for the uncertainty in estimating $\alpha$. Since the scale is to be estimated, its uncertainty also contributes to the limiting law of the calibrated median.

\begin{assumption}[Asymptotic linearity of the calibrated scale]\label{ass:alpha-if}
There exists a mean-zero square-integrable random variable $\phi_\alpha(Z)$ such that
\begin{equation}\label{eq:alpha-if}
 \sqrt n(\widehat\alpha_{\rm sc}-\alpha_{\rm sc})
 =\frac1{\sqrt n}\sum_{i=1}^n\phi_\alpha(Z_i)+o_{\Pbb}(1).
\end{equation}
\end{assumption}

Appendix \ref{app:scale-if} gives primitive conditions under which \eqref{eq:alpha-if} follows from asymptotic linearity of marginal medians and radial quantiles.

Let
\[
 m_0=m_{\alpha_{\rm sc}},
 \qquad
 A_0=A_{\alpha_{\rm sc}},
 \qquad
 B_0=B_{\alpha_{\rm sc}},
 \qquad
 \psi_0(Z)=\psi_{\alpha_{\rm sc}}(Z;m_0).
\]

The expansion below is the fixed-scale median expansion plus a correction term
from the influence function of $\widehat\alpha_{\rm sc}$.
\begin{theorem}[Two-step CLT for the scale-calibrated median]\label{thm:calibrated-clt}
Assume the conditions of Theorem \ref{thm:functional-clt} in a neighborhood of $\alpha_{\rm sc}$ and Assumption \ref{ass:alpha-if}.  Then
\begin{equation}\label{eq:calibrated-linearization}
 \sqrt n\log_{m_0}(\widehat m_{n,\widehat\alpha_{\rm sc}})
 =
 -A_0^{-1}\frac1{\sqrt n}\sum_{i=1}^n
 \left\{\psi_0(Z_i)+B_0\phi_\alpha(Z_i)\right\}
 +o_{\Pbb}(1).
\end{equation}
Consequently,
\begin{equation}\label{eq:calibrated-clt}
 \sqrt n\log_{m_0}(\widehat m_{n,\widehat\alpha_{\rm sc}})
 \Rightarrow
 N(0,V_{\rm sc}),
\end{equation}
where
\begin{equation}\label{eq:Vsc}
 V_{\rm sc}
 =
 A_0^{-1}
 \Var\{\psi_0(Z)+B_0\phi_\alpha(Z)\}
 A_0^{-T}.
\end{equation}
\end{theorem}
The additional term $B_0\phi_\alpha(Z)$ shows how marginal scale uncertainty is
transmitted through the sensitivity of the product median to $\alpha$.

\subsection{Local robustness}

The same expansion gives the influence function of the calibrated median:
\begin{equation}\label{eq:calibrated-if}
 \operatorname{IF}_{\rm sc}(Z)
 =-A_0^{-1}\{\psi_0(Z)+B_0\phi_\alpha(Z)\}.
\end{equation}
Since $\|\psi_0(Z)\|$ is bounded by $1/\sqrt\eps$ on $I_\eps$, bounded influence follows if $\phi_\alpha$ is bounded and $A_0^{-1}$ is finite.  Radial median scales have bounded influence under positive density at the median and standard smoothness conditions; hence the calibrated product median preserves the local robustness of the fixed-$\alpha$ product median.  In implementation, we recommend truncating
\[
 \widehat\alpha_{\rm sc}\leftarrow
 \min\{2-\eps,\max\{\eps,\widehat\alpha_{\rm sc}\}\}
\]
for a small fixed $\eps>0$, which prevents finite-sample metric collapse.

\subsection{Dimension-adjusted calibration}
When the factors have very different dimensions, one may standardize dispersion
per tangent dimension rather than total marginal dispersion.  Let
\[
 d_M^*=\dim M,
 \qquad
 d_N^*=\dim N.
\]
The dimension-adjusted calibrated weight is
\begin{equation}\label{eq:dimension-adjusted-alpha}
 \widehat\alpha_{\rm dim}
 =
 2\frac{d_M^*/\widehat s_M^2}
 {d_M^*/\widehat s_M^2+d_N^*/\widehat s_N^2}.
\end{equation}
All consistency and two-step CLT results above apply with the obvious modification of the map $(s_M,s_N)\mapsto\alpha$.

\section{Approach III: a balanced estimating equation for $\alpha$}\label{sec:balanced}

Scale calibration chooses $\alpha$ by marginal standardization.  We now give a
direct data-adaptive alternative: add a bounded estimating equation that balances
the two factor contributions.  This is not joint minimization of
$F_\alpha(m)$ over $\alpha$, which is degenerate, but an augmented
location--scale estimating system.

\subsection{The balance function}

The balance equation compares the two scaled squared factor distances through a
bounded relative contrast.  For $m=(p,q)$, define
\[
 A_m(Z)=d_M(p,X)^2,
 \qquad
 B_m(Z)=d_N(q,Y)^2.
\]
For $\alpha\in(0,2)$, set
\begin{equation}\label{eq:h-alpha}
 h_\alpha(Z;m)
 =
 \frac{\alpha A_m(Z)-(2-\alpha)B_m(Z)}
 {\alpha A_m(Z)+(2-\alpha)B_m(Z)}.
\end{equation}
Whenever the denominator is nonzero,
\[
 -1\le h_\alpha(Z;m)\le1.
\]
The numerator is the difference between the scaled squared factor distances, and the denominator is their sum.  Thus $h_\alpha$ measures relative contribution imbalance on a bounded scale.

The balanced population target $(m_0,\alpha_0)$ is defined by
\begin{equation}\label{eq:balanced-pop-eq}
 G_{\alpha_0}(m_0)=0,
 \qquad
 H_{\alpha_0}(m_0):=E[h_{\alpha_0}(Z;m_0)]=0.
\end{equation}
The sample estimator solves
\begin{equation}\label{eq:balanced-sample-eq}
 G_{n,\widehat\alpha}(\widehat m)=0,
 \qquad
 H_{n,\widehat\alpha}(\widehat m):=\frac1n\sum_{i=1}^n h_{\widehat\alpha}(Z_i;\widehat m)=0.
\end{equation}
The equation selects the scale at which the average relative imbalance between
the two factors is zero.

\subsection{Monotonicity and uniqueness for fixed location}

For fixed location, the scale equation is one-dimensional and monotone, which
gives both identification and a simple numerical update.

\begin{lemma}[Monotonicity of the balance equation]\label{lem:h-monotone}
For fixed $m$ and $Z$, let $A=A_m(Z)$ and $B=B_m(Z)$.  If $A+B>0$, then
\begin{equation}\label{eq:h-derivative}
 \frac{\partial}{\partial\alpha}
 \frac{\alpha A-(2-\alpha)B}{\alpha A+(2-\alpha)B}
 =
 \frac{4AB}{\{\alpha A+(2-\alpha)B\}^2}\ge0.
\end{equation}
Thus $\alpha\mapsto h_\alpha(Z;m)$ is nondecreasing, and it is strictly increasing whenever $A>0$ and $B>0$.
\end{lemma}
Increasing $\alpha$ increases the relative contribution of the $M$-factor and
decreases that of the $N$-factor, so the imbalance can only increase. The endpoint limits then force a unique interior zero under mild
nondegeneracy.
\begin{theorem}[Unique balancing scale for fixed location]\label{thm:unique-balance-fixed-m}
Fix $m=(p,q)$.  Assume
\[
 P(A_m(Z)>0)=1,
 \qquad
 P(B_m(Z)>0)=1,
 \qquad
 P(A_m(Z)B_m(Z)>0)>0.
\]
Then $\alpha\mapsto H_\alpha(m)$ is continuous and strictly increasing on $(0,2)$, with
\[
 \lim_{\alpha\downarrow0}H_\alpha(m)=-1,
 \qquad
 \lim_{\alpha\uparrow2}H_\alpha(m)=1.
\]
Consequently, there exists a unique $\alpha(m)\in(0,2)$ satisfying
\[
 H_{\alpha(m)}(m)=0.
\]
\end{theorem}

This is the constructive counterpart to the no-go theorem: minimization over
scale fails, but a separate monotone moment equation identifies an interior
scale.

\subsection{Joint consistency and asymptotic normality}

We now treat the median equation and the balance equation as a single
$Z$-estimation problem.  Let $\theta=(m,\alpha)$ and define the augmented
estimating function
\begin{equation}\label{eq:Xi}
 \Xi_\alpha(Z;m)
 =
 \begin{pmatrix}
 \psi_\alpha(Z;m)\\
 h_\alpha(Z;m)
 \end{pmatrix}.
\end{equation}
The first component is tangent-space valued and the second is scalar.  In local coordinates around $m_0$, $\Xi_\alpha(Z;m)$ is a vector in $\R^{d+1}$, where $d=\dim M+\dim N$.

The following assumption collects the usual uniqueness, smoothness, and
nonsingularity conditions for this augmented system.
\begin{assumption}[Balanced target regularity]\label{ass:balanced}
There exists a unique pair $(m_0,\alpha_0)\in K\times I_\eps$ satisfying
\[
 E[\Xi_{\alpha_0}(Z;m_0)]=0.
\]
The map $(m,\alpha)\mapsto E[\Xi_\alpha(Z;m)]$ is continuously differentiable in a neighborhood of $(m_0,\alpha_0)$.  Its Jacobian
\begin{equation}\label{eq:J-balanced}
 J=
 D_{(m,\alpha)}E[\Xi_\alpha(Z;m)]\big|_{(m_0,\alpha_0)}
\end{equation}
is nonsingular.  The class of functions
\[
 \{\Xi_\alpha(\cdot;m):(m,\alpha)\in U\}
\]
is $P$-Donsker for some neighborhood $U$ of $(m_0,\alpha_0)$.
\end{assumption}

Under these conditions, the estimator has the standard asymptotic linear
representation for regular estimating equations. This is a direct application of the usual local theory for regular
$Z$-estimators \citep{vandervaart_1996_WeakConvergenceEmpirical, vandervaart_1998_AsymptoticStatistics}.

\begin{theorem}[Balanced estimator: asymptotic linearity and CLT]\label{thm:balanced-clt}
Suppose Assumption \ref{ass:balanced} holds and let $(\widehat m,\widehat\alpha)$ be a measurable approximate solution of \eqref{eq:balanced-sample-eq} satisfying
\[
 \left\|P_n\Xi_{\widehat\alpha}(\cdot;\widehat m)\right\|=o_{\Pbb}(n^{-1/2})
\]
and $(\widehat m,\widehat\alpha)\to(m_0,\alpha_0)$ in probability.  Then
\begin{equation}\label{eq:balanced-linearization}
 \sqrt n
 \begin{pmatrix}
 \log_{m_0}\widehat m\\
 \widehat\alpha-\alpha_0
 \end{pmatrix}
 =
 -J^{-1}\frac1{\sqrt n}\sum_{i=1}^n\Xi_{\alpha_0}(Z_i;m_0)+o_{\Pbb}(1).
\end{equation}
Consequently,
\begin{equation}\label{eq:balanced-normal}
 \sqrt n
 \begin{pmatrix}
 \log_{m_0}\widehat m\\
 \widehat\alpha-\alpha_0
 \end{pmatrix}
 \Rightarrow
 N\left(0,
 J^{-1}\Var\{\Xi_{\alpha_0}(Z;m_0)\}J^{-T}
 \right).
\end{equation}
\end{theorem}

The theorem assumes consistency to focus on the asymptotic linearization.  Consistency follows from the standard $Z$-estimation argmin/zero theorem if $E[\Xi_\alpha(Z;m)]$ has a unique zero, the empirical process converges uniformly, and the approximate zeros are tight in a compact neighborhood.  These conditions are implied by Assumption \ref{ass:balanced} plus the compact local geometry imposed in Assumption \ref{ass:geometry}.

\subsection{Bounded influence}

The influence function follows directly from the linear representation:
\begin{equation}\label{eq:balanced-if}
 \operatorname{IF}_{\rm bal}(Z)
 =-J^{-1}\Xi_{\alpha_0}(Z;m_0).
\end{equation}
On $I_\eps$,
\[
 \|\psi_\alpha(Z;m)\|_{g_1}
 =
 \frac{\{A_m(Z)+B_m(Z)\}^{1/2}}
 {\{\alpha A_m(Z)+(2-\alpha)B_m(Z)\}^{1/2}}
 \le \frac1{\sqrt\eps},
\]
and
\[
 |h_\alpha(Z;m)|\le1.
\]
Thus $\Xi_{\alpha_0}(Z;m_0)$ is bounded.  If $J^{-1}$ is finite, then the balanced estimator has bounded influence.

The bounded contrast $h_\alpha$ is therefore preferable to an unnormalized
difference of squared distances, which would be more sensitive to extreme
observations.

\subsection{Dimension-adjusted balance}

As in scale calibration, the balance equation can be adjusted to account for
different factor dimensions.  For factors of unequal dimensions $d_M^*=\dim M$ and $d_N^*=\dim N$, a dimension-adjusted balance function is
\begin{equation}\label{eq:h-dim}
 h_\alpha^{\rm dim}(Z;m)
 =
 \frac{
 \alpha d_M(p,X)^2/d_M^*-(2-\alpha)d_N(q,Y)^2/d_N^*
 }
 {
 \alpha d_M(p,X)^2/d_M^*+(2-\alpha)d_N(q,Y)^2/d_N^*
 }.
\end{equation}
All monotonicity, uniqueness, consistency, CLT, and bounded-influence results carry over after replacing $A_m$ by $d_M(p,X)^2/d_M^*$ and $B_m$ by $d_N(q,Y)^2/d_N^*$.

The three approaches serve different purposes.  The path quantifies scale
sensitivity, calibration gives a unit-invariant default, and balancing provides
direct data-adaptive scale estimation.

\section{Numerical experiments}\label{sec:experiments}

We now illustrate the preceding theory in four simulation studies.  The first
experiment visualizes the boundary degeneracy of naive joint minimization and
the sensitivity path $\alpha\mapsto \widehat m_{n,\alpha}$.  The second
experiment studies the scale-calibrated estimator and its invariance to changes
of units.  The third experiment examines the balanced estimating equation.  The
fourth experiment applies the proposed procedures to Gaussian distributions
under the Bures--Wasserstein product geometry.

All computations were performed in Python.  Fixed-$\alpha$ product medians were
computed by the product Weiszfeld algorithm \citep{park_2026_GeometricMediansProduct}. 
For the balanced estimator, we alternated between a fixed-$\alpha$ median update
and a one-dimensional bisection step for the balance equation.  Unless otherwise
stated, the scale grid was contained in $[0.05,1.95]$.

\subsection{Boundary degeneracy and sensitivity path}

The first experiment uses the Euclidean product
\[
 M=\mathbb R^2,\qquad N=\mathbb R^2,
\]
with
\[
 d_\alpha((p,q),(x,y))^2
 =
 \alpha\|p-x\|^2+(2-\alpha)\|q-y\|^2 .
\]
The data are generated from a three-component mixture.  With probabilities
$0.45$, $0.30$, and $0.25$, respectively, observations are sampled from a
central component, a component shifted only in the $M$-factor, and a component
shifted only in the $N$-factor:
\[
 X=0.1\,\varepsilon_X,\qquad Y=0.1\,\varepsilon_Y
\]
for the central component,
\[
 X=(4,0)+0.1\,\varepsilon_X,\qquad Y=0.1\,\varepsilon_Y
\]
for the $M$-shifted component, and
\[
 X=0.1\,\varepsilon_X,\qquad Y=(3,0)+0.1\,\varepsilon_Y
\]
for the $N$-shifted component, where
$\varepsilon_X,\varepsilon_Y\sim N_2(0,I_2)$.  This design creates a simple
setting in which the two factors contain different geometric deviations.

For each sample, we computed the empirical path
$\alpha\mapsto \widehat m_{n,\alpha}$ and the profiled objective
\[
 \widehat\varphi_n(\alpha)
 =
 \min_m F_{n,\alpha}(m).
\]
Figure~\ref{fig:exp1-boundary-sensitivity} summarizes the representative path
and endpoint-selection behavior.  The profiled objective is minimized at an
endpoint of the scale grid, as predicted by Theorem~\ref{thm:boundary-degeneracy},
while the median path itself varies smoothly with $\alpha$.

\begin{figure}[ht]
\centering
\includegraphics[width=.95\textwidth]{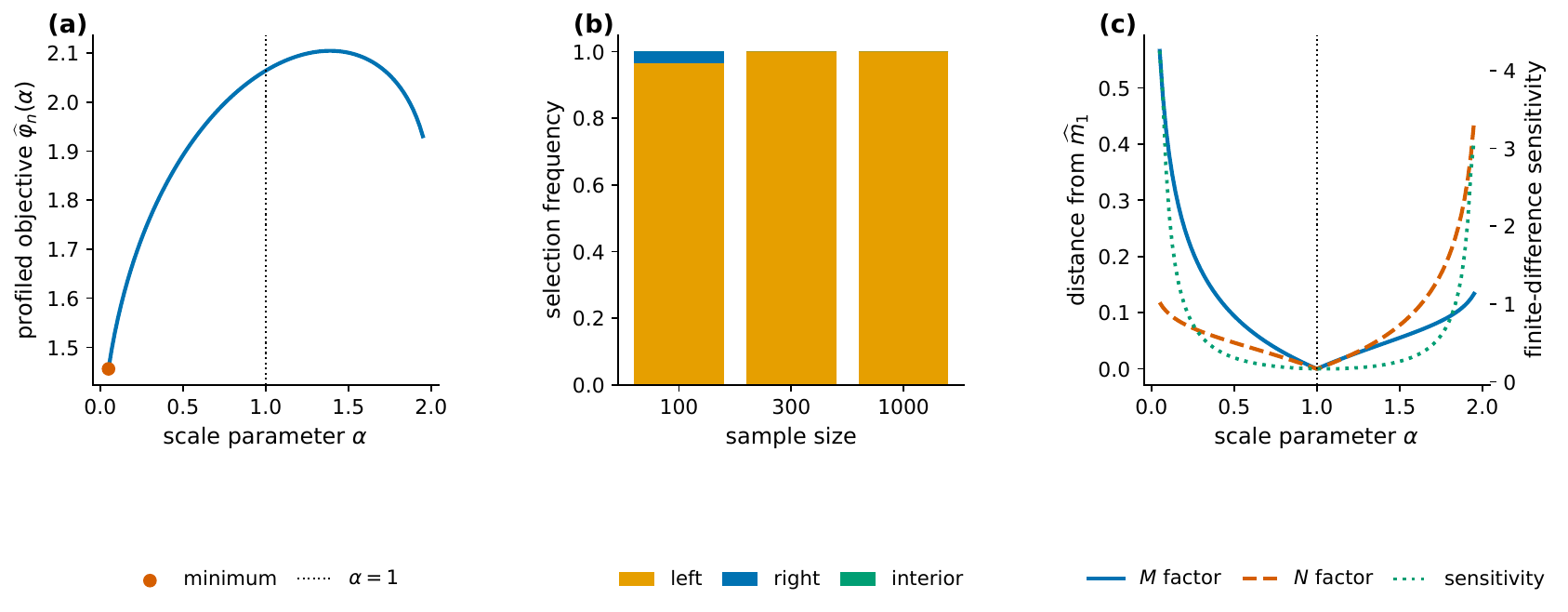}
\caption{
Boundary degeneracy and sensitivity path in the Euclidean product example.
(a) Profiled objective $\widehat\varphi_n(\alpha)$ for a representative sample,
with the minimum marked and $\alpha=1$ shown by the vertical reference line.
(b) Selection frequencies of the left endpoint, right endpoint, and interior
grid points under naive minimization over $\alpha$.
(c) Factorwise displacement from $\widehat m_{n,1}$ and finite-difference
sensitivity along the median path.
}
\label{fig:exp1-boundary-sensitivity}
\end{figure}

Table~\ref{tab:exp1-summary} summarizes the Monte Carlo study.  Across
$n=100,300,1000$, the profiled objective is smallest at the left endpoint on
average.  For $n=100$, the left endpoint is selected in $96.6\%$ of repetitions
and the right endpoint in $3.4\%$; for $n=300$ and $n=1000$, the left endpoint
is selected in all repetitions.  No repetition selects an interior value.  The
experiment therefore illustrates two distinct facts: the median path itself is
meaningful, but minimizing over the path is not a valid way to learn scale.

\begin{table}[t]
\centering
\small
\caption{
Experiment 1 summary.  The first three objective columns report Monte Carlo
averages of the profiled objective at the left endpoint, right endpoint, and
$\alpha=1$.  The last three columns report endpoint selection frequencies for
naive minimization over the scale grid.  The column ``conv.'' gives the
algorithmic convergence rate.
}
\label{tab:exp1-summary}
\begin{tabular}{rrrrrrrr}
\toprule
$n$
& $\bar\varphi(0.05)$
& $\bar\varphi(1.95)$
& $\bar\varphi(1)$
& conv.
& left
& right
& interior \\
\midrule
100  & 1.3756 & 1.9039 & 1.9981 & 1.000 & 0.966 & 0.034 & 0.000 \\
300  & 1.3689 & 1.9062 & 1.9976 & 1.000 & 1.000 & 0.000 & 0.000 \\
1000 & 1.3713 & 1.8987 & 1.9951 & 1.000 & 1.000 & 0.000 & 0.000 \\
\bottomrule
\end{tabular}
\end{table}

\subsection{Robust scale calibration and unit invariance}

The second experiment studies the effect of changing measurement units in one
factor.  We use the same Euclidean product model as in the first experiment,
with $n=300$ and $500$ Monte Carlo replications.  The $M$-factor metric is
rescaled by
\[
 d_M^{(c)}=c\,d_M,
 \qquad
 c\in\{0.1,0.25,0.5,1,2,4,10\}.
\]
For each value of $c$, we compute three estimators: the fixed-scale product
median with $\alpha=1$, the robust scale-calibrated median, and the balanced
estimator.

The drift from the reference estimate at $c=1$ is measured by
\[
 \Delta(c)
 =
 \left\{
 \|\widehat p^{(c)}-\widehat p^{(1)}\|^2
 +
 \|\widehat q^{(c)}-\widehat q^{(1)}\|^2
 \right\}^{1/2}.
\]
Figure~\ref{fig:exp2-scale-calibration} shows that the fixed $\alpha=1$ median
changes substantially under unit rescaling, while the scale-calibrated and
balanced estimators remain invariant up to numerical precision.  The selected
values of $\widehat\alpha$ vary with $c$, but their effective original-coordinate
weight ratios remain stable for the adaptive methods.  This confirms the scale
equivariance result in Theorem~\ref{thm:scale-equivariance}.

\begin{figure}[t]
\centering
\includegraphics[width=.95\textwidth]{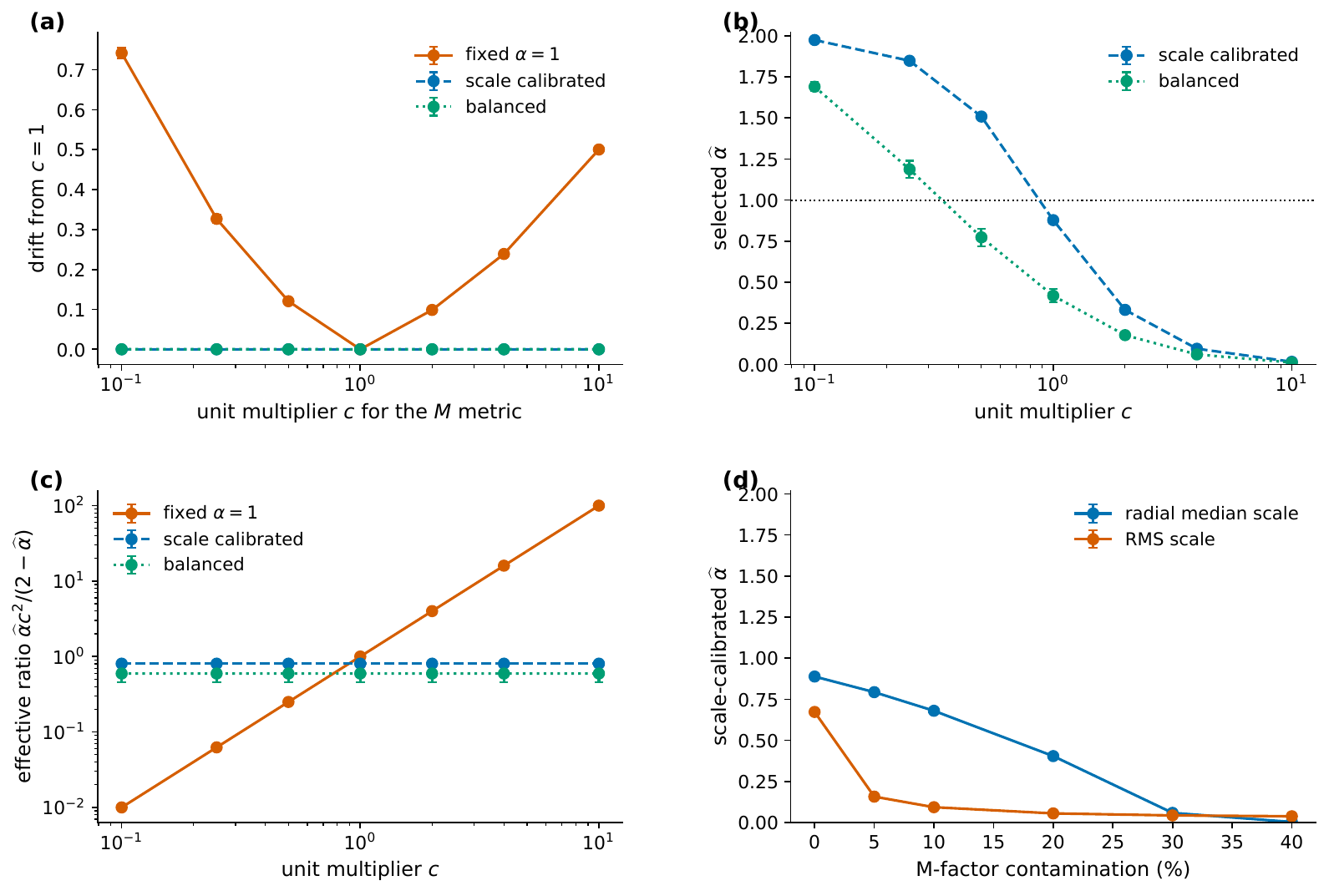}
\caption{
Scale calibration and unit invariance.
(a) Location drift from the estimate at $c=1$ after multiplying the $M$-factor
metric by $c$.
(b) Selected scale values for the calibrated and balanced procedures.
(c) Effective original-coordinate weight ratio
$\widehat\alpha c^2/(2-\widehat\alpha)$ under unit rescaling.
(d) Scale-calibrated $\widehat\alpha$ under $M$-factor contamination, comparing
radial-median and RMS scale estimates.
}
\label{fig:exp2-scale-calibration}
\end{figure}

Table~\ref{tab:exp2-unit} gives the corresponding numerical summaries.  For the
fixed $\alpha=1$ median, the drift is $0.7419$ at $c=0.1$ and $0.5003$ at
$c=10$.  In contrast, both adaptive estimators have zero drift to the displayed
precision.  Their selected values of $\alpha$ compensate for the metric
rescaling: as $c$ increases, the reported $\widehat\alpha$ decreases.

\begin{table}[t]
\centering
\small
\caption{
Experiment 2 unit-rescaling summary.  Drifts are measured relative to the
corresponding estimate at $c=1$.
}
\label{tab:exp2-unit}
\begin{tabular}{rrrrrr}
\toprule
$c$
& fixed drift
& calibrated drift
& balanced drift
& $\widehat\alpha_{\rm sc}$
& $\widehat\alpha_{\rm bal}$ \\
\midrule
0.10 & 0.7419 & 0.0000 & 0.0000 & 1.9738 & 1.6898 \\
0.25 & 0.3267 & 0.0000 & 0.0000 & 1.8477 & 1.1871 \\
0.50 & 0.1204 & 0.0000 & 0.0000 & 1.5082 & 0.7733 \\
1.00 & 0.0000 & 0.0000 & 0.0000 & 0.8779 & 0.4179 \\
2.00 & 0.0986 & 0.0000 & 0.0000 & 0.3319 & 0.1780 \\
4.00 & 0.2387 & 0.0000 & 0.0000 & 0.0954 & 0.0605 \\
10.00 & 0.5003 & 0.0000 & 0.0000 & 0.0159 & 0.0114 \\
\bottomrule
\end{tabular}
\end{table}

We also compare robust radial-median scale calibration with RMS-based scale
calibration under contamination in the $M$-factor.  A fraction $\eta$ of the
sample is replaced by observations with
\[
 X^{\rm out}=(20,0)+\varepsilon,
 \qquad
 Y^{\rm out}=Y,
 \qquad
 \varepsilon\sim N_2(0,I_2).
\]
Table~\ref{tab:exp2-contamination} shows that the RMS scale reacts sharply even
to $5\%$ contamination, driving the calibrated $\widehat\alpha$ from $0.6727$ to
$0.1574$.  The radial-median scale is more stable for small and moderate
contamination, although it also deteriorates as contamination approaches the
breakdown regime.  This supports the use of robust marginal scales in the
calibration step.

\begin{table}[t]
\centering
\small
\caption{
Experiment 2 contamination summary.  The table reports the scale-calibrated
$\widehat\alpha$ obtained from robust radial-median scales and from RMS scales.
}
\label{tab:exp2-contamination}
\begin{tabular}{rrr}
\toprule
contamination $\eta$
& radial-median scale
& RMS scale \\
\midrule
0.00 & 0.8882 & 0.6727 \\
0.05 & 0.7932 & 0.1574 \\
0.10 & 0.6798 & 0.0931 \\
0.20 & 0.4040 & 0.0553 \\
0.30 & 0.0585 & 0.0431 \\
0.40 & 0.0031 & 0.0379 \\
\bottomrule
\end{tabular}
\end{table}

\subsection{Balanced estimating equation}

The third experiment evaluates the balanced estimator from
Section~\ref{sec:balanced}.  We again use the Euclidean product mixture from
Experiment 1.  The population target $(m_0,\alpha_0)$ is approximated by a large
Monte Carlo sample of size $200000$.  We then generate samples of size
\[
 n\in\{100,200,500,1000\}
\]
with $1000$ replications for each $n$.  For each sample, the balanced estimator
is computed by alternating fixed-scale median updates and bisection for the
balance equation.

Figure~\ref{fig:exp3-balanced} gives diagnostics for the balanced estimator.
The empirical balance function is monotone and crosses zero at the selected
scale, while the Monte Carlo diagnostics show the finite-sample behavior of the
scale and location components.

\begin{figure}[t]
\centering
\includegraphics[width=.95\textwidth]{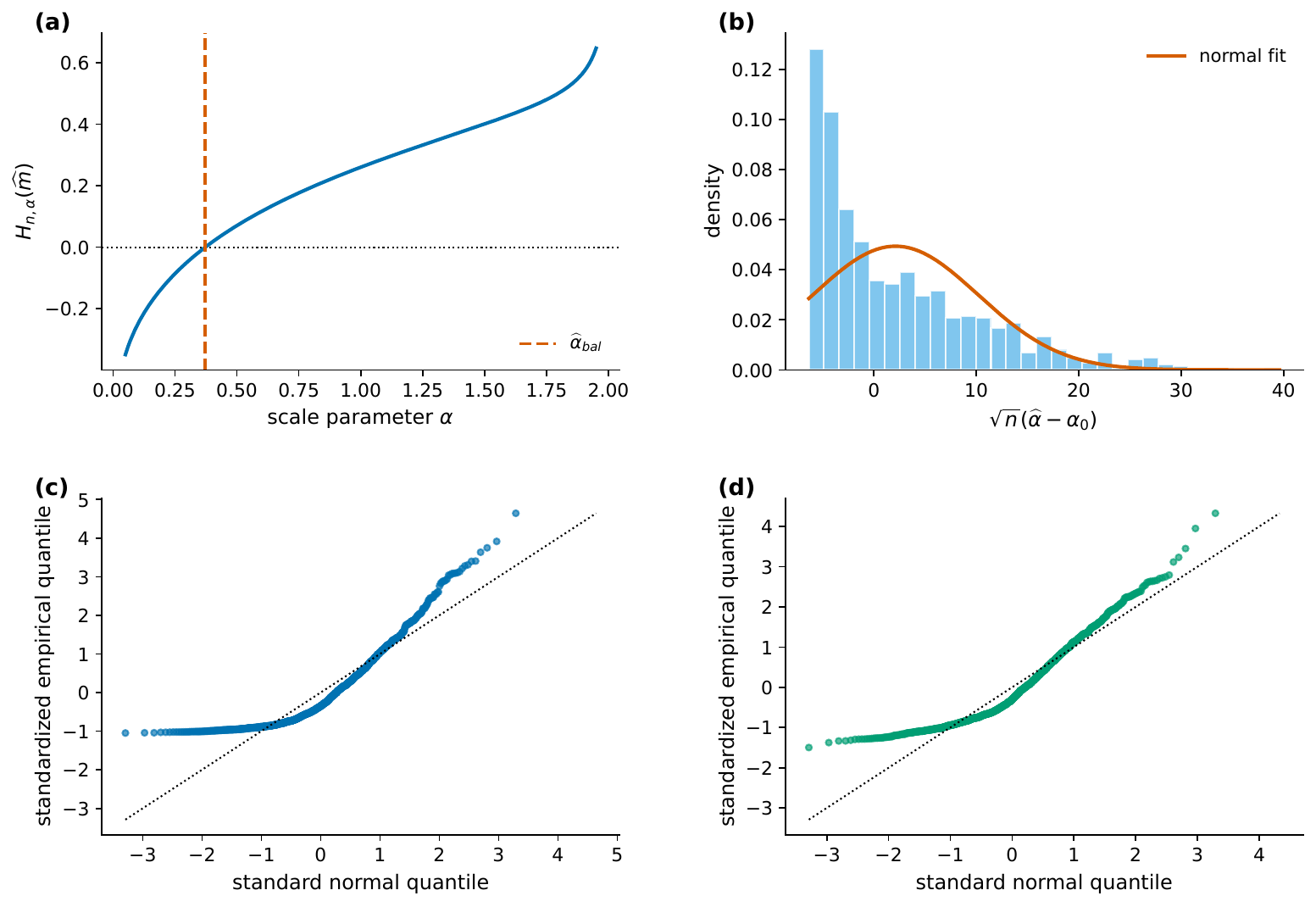}
\caption{
Balanced estimating equation.
(a) Empirical balance function for a representative sample, with the selected
$\widehat\alpha_{\rm bal}$ marked.
(b) Monte Carlo distribution of
$\sqrt n(\widehat\alpha-\alpha_0)$ for $n=1000$, with a fitted normal density.
(c) Normal Q--Q plot for the standardized scale estimator.
(d) Normal Q--Q plot for a standardized scalar projection of the location
estimator.
}
\label{fig:exp3-balanced}
\end{figure}

Table~\ref{tab:exp3-balanced} reports the Monte Carlo summaries.  The bias and
RMSE of $\widehat\alpha$ decrease with $n$.  The quantity
$\sqrt n\,\mathrm{sd}(\widehat\alpha)$ is reasonably stable, increasing from
$6.6765$ at $n=100$ to $8.0733$ at $n=1000$, which is consistent with
$n^{-1/2}$ scaling.  The empirical convergence rate of the algorithm is above
$96\%$ for all sample sizes and reaches $99.2\%$ at $n=1000$.  

The Wald intervals based on the sandwich standard error are not yet close to
nominal coverage in this design, with coverage increasing from $0.572$ at
$n=100$ to $0.752$ at $n=1000$.  This is in line with the skewness seen in
Figure~\ref{fig:exp3-balanced}.  In applications, bootstrap intervals may
therefore be preferable for the scale component when sample sizes are moderate.

\begin{table}[t]
\centering
\small
\caption{
Experiment 3 balanced-estimator asymptotics.  The target $\alpha_0$ is
approximated by a large Monte Carlo sample.  The column ``conv.'' gives the
algorithmic convergence rate.
}
\label{tab:exp3-balanced}
\begin{tabular}{rrrrrrrr}
\toprule
$n$
& bias
& sd
& $\sqrt n\,\mathrm{sd}$
& RMSE
& mean SE
& coverage
& conv. \\
\midrule
100  & 0.3879 & 0.6676 & 6.6765 & 0.7719 & 0.5427 & 0.572 & 0.967 \\
200  & 0.2307 & 0.5531 & 7.8221 & 0.5990 & 0.4365 & 0.599 & 0.964 \\
500  & 0.1251 & 0.3557 & 7.9532 & 0.3769 & 0.3297 & 0.708 & 0.971 \\
1000 & 0.0668 & 0.2553 & 8.0733 & 0.2638 & 0.2445 & 0.752 & 0.992 \\
\bottomrule
\end{tabular}
\end{table}

\subsection{Bures--Wasserstein Gaussian product example}

The final experiment studies a non-Euclidean product arising from Gaussian
distributions.  A Gaussian law $N_d(\mu,\Sigma)$ is represented by its mean
$\mu\in\mathbb R^d$ and covariance matrix $\Sigma\in\mathcal S_{++}^d$.  The
squared 2-Wasserstein distance between Gaussian distributions decomposes as
\[
 W_2^2\{N_d(\mu_1,\Sigma_1),N_d(\mu_2,\Sigma_2)\}
 =
 \|\mu_1-\mu_2\|^2
 +
 d_{\rm BW}(\Sigma_1,\Sigma_2)^2,
\]
where
\[
 d_{\rm BW}(\Sigma_1,\Sigma_2)^2
 =
 \operatorname{tr}\{\Sigma_1+\Sigma_2
 -2(\Sigma_2^{1/2}\Sigma_1\Sigma_2^{1/2})^{1/2}\}.
\]
This is the Gaussian optimal transport distance and is closely related to the
Bures--Wasserstein geometry of positive-definite matrices \citep{dowson_1982_FrechetDistanceMultivariate, takatsu_2011_WassersteinGeometryGaussian}. We therefore use the scaled product distance
\[
 d_\alpha^2
 =
 \alpha\|\mu-\mu_i\|^2
 +(2-\alpha)d_{\rm BW}(\Sigma,\Sigma_i)^2 .
\]

The data are generated from a three-component mixture.  The central component
has mean near zero and covariance near $I_d$.  The mean-shift component has mean
shift $\delta e_1$ with $\delta=3$ and covariance near $I_d$.  The covariance
shift component has mean near zero and covariance near the AR(1) matrix
\[
 \Sigma_{\rm AR}(\rho)_{jk}=\rho^{|j-k|},
 \qquad \rho=0.7 .
\]
Small random perturbations are added to both means and covariances.  The mixture
weights are $0.50$, $0.25$, and $0.25$.

Figure~\ref{fig:exp4-bw} shows one visualization with $d=2$ and $n=250$.
The scale-calibrated and balanced choices are close in this representative
dataset, with selected values approximately $\widehat\alpha_{\rm sc}=0.47$ and
$\widehat\alpha_{\rm bal}=0.49$.

\begin{figure}[t]
\centering
\includegraphics[width=.95\textwidth]{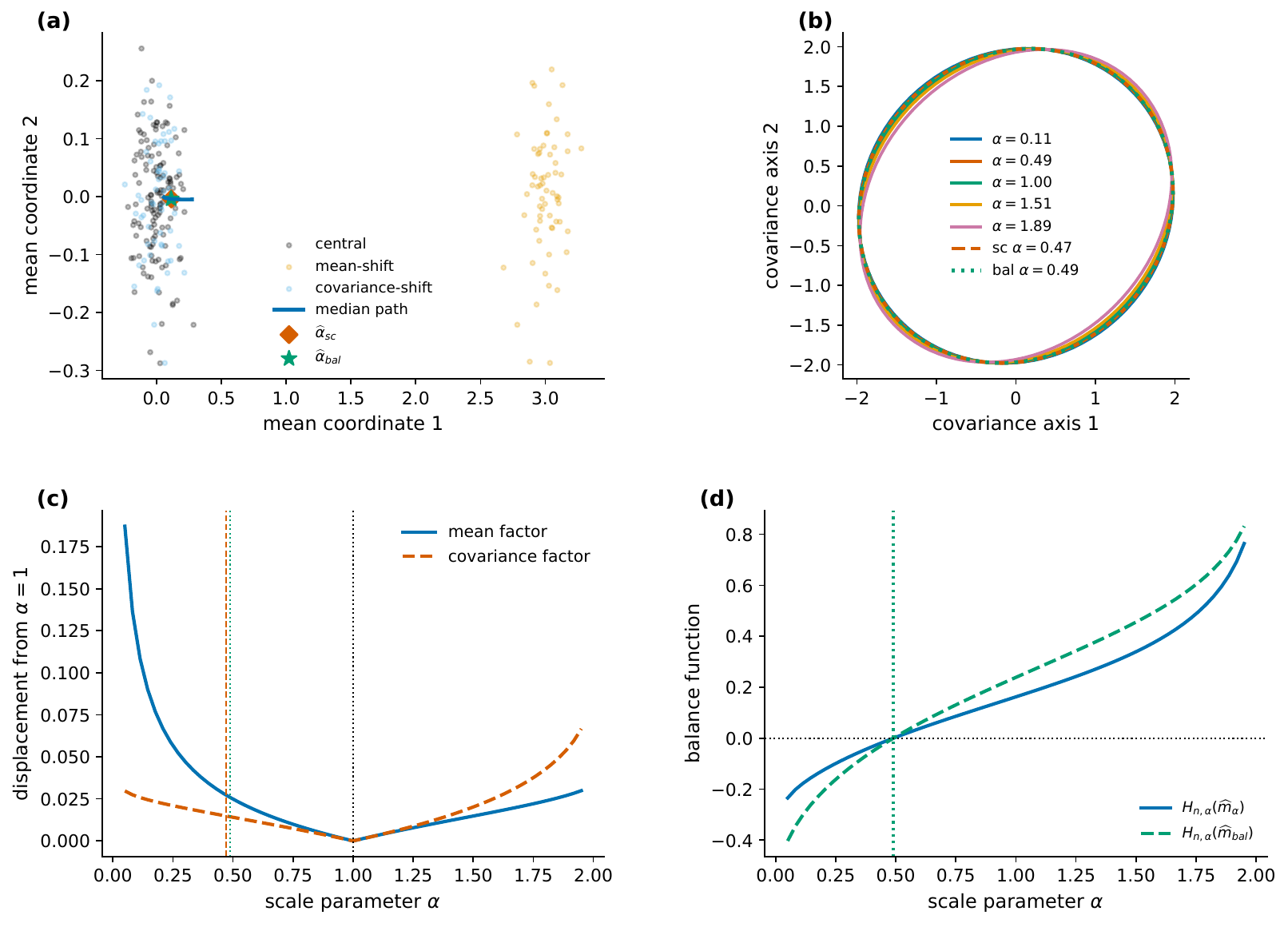}
\caption{
Bures--Wasserstein Gaussian product example.
(a) Mean observations and mean-factor trajectory of the median path.
(b) Covariance ellipses along the median path, with scale-calibrated and
balanced estimates overlaid.
(c) Mean-factor and covariance-factor displacements from the fixed-scale median
at $\alpha=1$.
(d) Balance functions evaluated along the fixed-scale path and at the balanced
median, with the selected balanced scale marked.
}
\label{fig:exp4-bw}
\end{figure}

We also ran a Monte Carlo study with $d=10$ and
$n\in\{100,300,1000\}$ using $300$ replications.  Table~\ref{tab:exp4-bw}
summarizes the selected scale values and the displacement of the adaptive
medians from the fixed $\alpha=1$ median.  Both methods converge in all
replications.  The selected scales are close to one on average in this
higher-dimensional design.  The scale-calibrated estimator has smaller
Monte Carlo variability in $\widehat\alpha$ than the balanced estimator, as
expected from its marginal plug-in construction.  The factorwise displacements
from the fixed $\alpha=1$ median are small and decrease with sample size.

\begin{table}[t]
\centering
\small
\caption{
Experiment 4 Bures--Wasserstein Monte Carlo summary for $d=10$.  Entries are
Monte Carlo means with standard deviations in parentheses.  The column
``conv.'' gives the algorithmic convergence rate.
}
\label{tab:exp4-bw}
\begin{tabular}{rlrrrr}
\toprule
$n$
& method
& $\widehat\alpha$
& mean displacement
& covariance displacement
& conv. \\
\midrule
100
& balanced
& $1.0894$ $(0.3004)$
& $0.0458$ $(0.0551)$
& $0.0324$ $(0.0301)$
& 1.000 \\
100
& scale calibrated
& $1.1278$ $(0.0839)$
& $0.0185$ $(0.0100)$
& $0.0158$ $(0.0125)$
& 1.000 \\
300
& balanced
& $1.0869$ $(0.1918)$
& $0.0273$ $(0.0217)$
& $0.0204$ $(0.0166)$
& 1.000 \\
300
& scale calibrated
& $1.1202$ $(0.0483)$
& $0.0174$ $(0.0063)$
& $0.0134$ $(0.0067)$
& 1.000 \\
1000
& balanced
& $1.1252$ $(0.0968)$
& $0.0198$ $(0.0103)$
& $0.0157$ $(0.0100)$
& 1.000 \\
1000
& scale calibrated
& $1.1281$ $(0.0244)$
& $0.0187$ $(0.0031)$
& $0.0140$ $(0.0034)$
& 1.000 \\
\bottomrule
\end{tabular}
\end{table}

Taken together, the four experiments support the main theoretical messages of
the paper.  Naive joint minimization over $\alpha$ degenerates to the boundary.
The fixed-scale median path is informative and can reveal scale sensitivity.
Robust scale calibration removes arbitrary unit dependence.  Finally, direct
data-adaptive scale estimation is possible when the median equation is
augmented with a separate balance condition.

\section{Conclusion}\label{sec:conclusion}

This paper studied relative metric scaling for geometric medians on product
manifolds.  The main message is that scale is target-defining.  Unlike the
Fr\'echet mean objective, the product-median objective is not separable across
factors, so changing the relative factor weights changes the population median.

The no-go theorem shows that scale cannot be learned by naive joint minimization
of the raw median objective.  For every fixed location, the objective is concave
in the scale parameter, and joint minimization degenerates to the boundary.  In
the limiting open-interval problem, the solution collapses to one of the
marginal median objectives.  Thus a meaningful scale choice requires an
additional identifying principle.

We developed three such principles.  The sensitivity-path approach treats the
family $\{m_\alpha:\alpha\in I_\eps\}$ as the object of inference and provides
uniform consistency, functional asymptotic normality, and a derivative-based
measure of scale sensitivity.  Robust scale calibration supplies a default
unit-invariant product median by standardizing each factor distance with a
marginal radial median scale.  The balanced estimating equation provides a
direct data-adaptive scale estimator by adding a bounded moment condition that
balances the two factor contributions.  These approaches serve complementary
purposes: path analysis quantifies scale dependence, calibration gives a robust
default, and balancing yields an interior learned scale.

The numerical experiments support the theory.  Naive joint minimization selects
endpoints rather than meaningful interior scales.  The empirical median path
reveals how the two factor components respond to changes in $\alpha$.  Robust
scale calibration removes artificial dependence on measurement units, and the
balanced estimating equation produces regular data-adaptive estimates, although
finite-sample normal approximation for the scale component may be slow in
imbalanced mixtures.  The Bures--Wasserstein Gaussian example shows that the
same phenomena occur in a non-Euclidean multivariate setting.

Several directions remain open.  First, the present theory is stated for two
factor manifolds, although the no-go phenomenon and the proposed principles
extend naturally to finite products.  A systematic treatment of many-factor
products would be useful for multimodal and tensor-valued data.  Second,
finite-sample inference for the sensitivity path and the balanced scale
estimator deserves further study, especially under contamination and high
dimension.  Third, stochastic and streaming algorithms for calibrated and
balanced product medians would broaden applicability to large-scale data.
Finally, many applications involve spaces that are not strict products, such as
fiber bundles or quotient geometries.  Extending scale-calibrated robust
location theory to these settings is an important topic for future work.


\appendix
\section{Auxiliary empirical-process details}\label{app:auxiliary}

This appendix records one set of primitive conditions implying Assumption \ref{ass:empirical}.  Suppose $K$ is contained in a finite union of normal-coordinate charts and the cut-locus condition in Assumption \ref{ass:no-singular} holds.  On $I_\eps\times K$, the distance class
\[
 \{z\mapsto d_\alpha(m,z):\alpha\in I_\eps,m\in K\}
\]
is Lipschitz in $(\alpha,m)$ with an integrable envelope of the form
\[
 C\{1+d_1(m_*,z)\}
\]
for any fixed $m_*\in K$.  Since $I_\eps\times K$ is totally bounded, the class is Glivenko--Cantelli.

For the score class, away from the diagonal and cut locus, the map
\[
 (\alpha,m)\mapsto \psi_\alpha(z;m)
\]
is continuously differentiable.  On $I_\eps$,
\[
 \|\psi_\alpha(z;m)\|_{g_1}
 \le \frac1{\sqrt\eps}.
\]
If the derivative with respect to $(\alpha,m)$ is bounded by a square-integrable envelope $L(z)$, then the score class is Lipschitz-parametric:
\[
 \|\psi_\alpha(z;m)-\psi_\beta(z;m')\|
 \le L(z)(|\alpha-\beta|+d_1(m,m')).
\]
Finite-dimensional Lipschitz-parametric classes with square-integrable envelopes are Donsker.  This gives Assumption \ref{ass:empirical}.

\section{Influence functions for radial median scales}\label{app:scale-if}

This appendix gives primitive sufficient conditions for Assumption \ref{ass:alpha-if}.  We present the $M$-factor; the $N$-factor is identical. The argument is a standard Bahadur-type expansion for sample quantiles, combined with the influence-function expansion of the marginal median \citep{bahadur_1966_NoteQuantilesLarge, serfling_1980_ApproximationTheoremsMathematical}. 

Let $p_0$ be the marginal median and $s_M$ the median of $R_M=d_M(p_0,X)$.  Define
\[
 H_M(p,s)=P\{d_M(p,X)\le s\}-\frac12.
\]
Assume $H_M(p_0,s_M)=0$, the distribution of $R_M$ has density $f_M$ with $f_M(s_M)>0$, and $p\mapsto H_M(p,s_M)$ is differentiable at $p_0$ with derivative $D_pH_M$.

Assume the marginal median has expansion
\[
 \sqrt n\log_{p_0}\widehat p_n
 =\frac1{\sqrt n}\sum_{i=1}^n\phi_p(X_i)+o_{\Pbb}(1).
\]
Then the radial median scale satisfies
\begin{equation}\label{eq:scale-if-M}
 \sqrt n(\widehat s_M-s_M)
 =\frac1{\sqrt n}\sum_{i=1}^n\phi_{s_M}(X_i)+o_{\Pbb}(1),
\end{equation}
where
\begin{equation}\label{eq:phisM}
 \phi_{s_M}(X)
 =
 \frac{1/2-\1\{d_M(p_0,X)\le s_M\}}{f_M(s_M)}
 -
 \frac{D_pH_M[\phi_p(X)]}{f_M(s_M)}.
\end{equation}
The sign convention depends on whether $\phi_p$ is defined as the influence function of $\widehat p_n$ or of the estimating equation; \eqref{eq:phisM} uses the convention displayed above.

Similarly,
\[
 \sqrt n(\widehat s_N-s_N)
 =\frac1{\sqrt n}\sum_{i=1}^n\phi_{s_N}(Y_i)+o_{\Pbb}(1).
\]
For
\[
 h(s_M,s_N)=2\frac{s_N^2}{s_M^2+s_N^2},
\]
we have
\[
 \frac{\partial h}{\partial s_M}
 =-\frac{4s_Ms_N^2}{(s_M^2+s_N^2)^2},
 \qquad
 \frac{\partial h}{\partial s_N}
 =\frac{4s_Ns_M^2}{(s_M^2+s_N^2)^2}.
\]
Thus
\begin{equation}\label{eq:alpha-if-formula}
 \phi_\alpha(Z)
 =
 -\frac{4s_Ms_N^2}{(s_M^2+s_N^2)^2}\phi_{s_M}(X)
 +
 \frac{4s_Ns_M^2}{(s_M^2+s_N^2)^2}\phi_{s_N}(Y).
\end{equation}
This yields Assumption \ref{ass:alpha-if}.

\section{Proofs}

\noindent \textbf{Proof of Lemma~\ref{lem:scaled-geometry}.}
If $g$ is replaced by $cg$ with $c>0$, the curvature tensor as a $(1,3)$ tensor is unchanged, while sectional curvature is divided by $c$.  Therefore $\secop_{\alpha g_M}=\alpha^{-1}\secop_{g_M}$ and $\secop_{(2-\alpha)g_N}=(2-\alpha)^{-1}\secop_{g_N}$.  The product curvature is bounded above by the larger of the factorwise upper bounds, since mixed planes are flat and arbitrary two-planes decompose as weighted combinations of factorwise planes.  The injectivity-radius scaling follows because lengths are multiplied by $\sqrt c$ under the scaling $g\mapsto cg$, whereas the geodesic paths and cut times as affine geodesics are unchanged.  The uniform bounds on $I_\eps$ are immediate.

\medskip 
\noindent \textbf{Proof of Lemma~\ref{lem:concavity}.}
The derivative formulas follow by direct differentiation of
\[
 \{\alpha A+(2-\alpha)B\}^{1/2}
\]
with $A=A_m(z)$ and $B=B_m(z)$.  The second derivative is nonpositive.  Sums and expectations of concave functions are concave.  Points at which $d_\alpha(m,z)=0$ are handled by continuity, or equivalently by the fact that the square-root of a nonnegative affine function is concave.

\medskip
\noindent \textbf{Proof of Theorem~\ref{thm:boundary-degeneracy}.}
For fixed $m$, Lemma \ref{lem:concavity} implies that the infimum of $\alpha\mapsto F_{n,\alpha}(m)$ over $[a,b]$ is attained at an endpoint:
\[
 \inf_{\alpha\in[a,b]}F_{n,\alpha}(m)=\min\{F_{n,a}(m),F_{n,b}(m)\}.
\]
Therefore
\[
 \inf_{m,\alpha\in[a,b]}F_{n,\alpha}(m)
 =\inf_m\min\{F_{n,a}(m),F_{n,b}(m)\}.
\]
For arbitrary functions $f$ and $g$,
\[
 \inf_m\min\{f(m),g(m)\}
 =\min\{\inf_m f(m),\inf_m g(m)\}.
\]
Applying this identity with $f=F_{n,a}$ and $g=F_{n,b}$ gives \eqref{eq:emp-boundary-degeneracy}.  The population proof is identical.

\medskip 
\noindent \textbf{Proof of Corollary~\ref{cor:endpoint-collapse}.}
For fixed $(p,q)$,
\[
 \lim_{\alpha\downarrow0}F_\alpha(p,q)=\sqrt2\,E[d_N(q,Y)],
\]
and
\[
 \lim_{\alpha\uparrow2}F_\alpha(p,q)=\sqrt2\,E[d_M(p,X)],
\]
by monotone or dominated convergence after using the elementary bound
\[
 d_\alpha((p,q),(X,Y))\le \sqrt2\{d_M(p,X)+d_N(q,Y)\}.
\]
By Theorem \ref{thm:boundary-degeneracy} applied to intervals $[a,b]\subset(0,2)$ and then taking $a\downarrow0$, $b\uparrow2$, the joint infimum over $(0,2)$ is the smaller of the two limiting endpoint infima.  This gives \eqref{eq:endpoint-collapse}.

\medskip 
\noindent \textbf{Proof of Corollary~\ref{cor:kfactor}.}
The map $w\mapsto \sum_jw_jd_j(m_j,z_j)^2$ is affine and nonnegative, and $t\mapsto\sqrt t$ is concave and increasing.  Their composition is concave.  Sums and expectations preserve concavity.  A concave function on a compact convex polytope attains its minimum at an extreme point, possibly nonuniquely.

\medskip 
\noindent \textbf{Proof of Proposition~\ref{prop:path-continuity}.}
The map $(\alpha,m)\mapsto F_\alpha(m)$ is continuous on the compact set $I_\eps\times K$.  Let $\alpha_k\to\alpha$.  By compactness, every subsequence of $m_{\alpha_k}$ has a further subsequence converging to some $m\in K$.  Along such a subsequence,
\[
 F_\alpha(m)=\lim_kF_{\alpha_k}(m_{\alpha_k})
 \le \lim_kF_{\alpha_k}(u)=F_\alpha(u)
\]
for all $u\in K$.  Hence $m$ is a minimizer of $F_\alpha$ on $K$.  By uniqueness, $m=m_\alpha$.  Thus every convergent subsequence has limit $m_\alpha$, and $m_{\alpha_k}\to m_\alpha$.

\medskip 
\noindent \textbf{Proof of Theorem~\ref{thm:uniform-consistency}.}
By Assumption \ref{ass:empirical},
\[
 \sup_{\alpha\in I_\eps,m\in K}|F_{n,\alpha}(m)-F_\alpha(m)|\to0.
\]
Fix $\eta>0$.  By uniform separation \eqref{eq:uniform-separation}, there is $c_\eta>0$ such that
\[
 F_\alpha(m)-F_\alpha(m_\alpha)\ge c_\eta
\]
whenever $d_1(m,m_\alpha)\ge\eta$.  With probability tending to one,
\[
 \sup_{\alpha,m}|F_{n,\alpha}(m)-F_\alpha(m)|<c_\eta/3
\]
and the approximate argmin error is smaller than $c_\eta/3$.  On this event, if $d_1(\widehat m_{n,\alpha},m_\alpha)\ge\eta$ for some $\alpha$, then
\[
 F_{n,\alpha}(\widehat m_{n,\alpha})
 \ge F_\alpha(\widehat m_{n,\alpha})-c_\eta/3
 \ge F_\alpha(m_\alpha)+2c_\eta/3
 \ge F_{n,\alpha}(m_\alpha)+c_\eta/3,
\]
contradicting the approximate minimizer property.  Hence the supremum distance is less than $\eta$ with probability tending to one.

\medskip
\noindent \textbf{Proof of Theorem~\ref{thm:functional-clt}.}
Let
$
 v_{n,\alpha}=\log_{m_\alpha}(\widehat m_{n,\alpha})\in T_{m_\alpha}\M.
$
Uniform consistency implies $\sup_\alpha\|v_{n,\alpha}\|\to0$ in probability.  The estimating equation gives
\[
 0=G_{n,\alpha}(\widehat m_{n,\alpha})
 =G_{n,\alpha}(m_\alpha)+A_\alpha v_{n,\alpha}+R_{n,\alpha},
\]
where the uniform differentiability and stochastic equicontinuity assumptions give
\[
 \sup_{\alpha\in I_\eps}\sqrt n\|R_{n,\alpha}\|=o_{\Pbb}(1).
\]
Moreover,
\[
 \sqrt n\,G_{n,\alpha}(m_\alpha)
 =\frac1{\sqrt n}\sum_{i=1}^n\psi_\alpha(Z_i;m_\alpha).
\]
Parallel transporting to $T_{m_{\alpha_0}}\M$ yields
\[
 0=
 \mathbb Z_n(\alpha)
 +\widetilde A_\alpha\sqrt n\,\Pi_\alpha v_{n,\alpha}
 +o_{\Pbb}(1)
\]
uniformly in $\alpha$.  Since $\sup_\alpha\|\widetilde A_\alpha^{-1}\|<\infty$, \eqref{eq:uniform-linear-expansion} follows.  Weak convergence \eqref{eq:functional-clt} follows from the functional CLT for $\mathbb Z_n$ and the continuous mapping theorem, because $\alpha\mapsto \widetilde A_\alpha^{-1}$ is continuous on compact $I_\eps$.

\medskip
\noindent \textbf{Proof of Theorem~\ref{thm:path-derivative}.}
The population first-order condition is
\[
 G_\alpha(m_\alpha)=0.
\]
By Assumption \ref{ass:differentiability}, the map $(\alpha,m)\mapsto G_\alpha(m)$ is continuously differentiable, and $D_mG_\alpha(m_\alpha)=A_\alpha$ is invertible.  The implicit function theorem on manifolds, applied in normal coordinates around each $m_\alpha$, implies that $\alpha\mapsto m_\alpha$ is continuously differentiable and
\[
 A_\alpha\dot m_\alpha+B_\alpha=0.
\]
This gives \eqref{eq:path-derivative}.

To compute $B_\alpha$, write $w_m(Z)=\log_mZ$ and $r_\alpha=d_\alpha(m,Z)$.  For fixed $m$, $w_m(Z)$ does not depend on $\alpha$, while
\[
 \partial_\alpha r_\alpha(m,Z)
 =
 \frac{d_M(p,X)^2-d_N(q,Y)^2}{2r_\alpha(m,Z)}.
\]
Thus
\[
 \partial_\alpha\left\{\frac{w_m(Z)}{r_\alpha(m,Z)}\right\}
 =
 -\frac{d_M(p,X)^2-d_N(q,Y)^2}{2r_\alpha(m,Z)^3}w_m(Z).
\]
Dominated differentiation is justified by Assumption \ref{ass:differentiability}, yielding \eqref{eq:Balpha}.

\medskip
\noindent \textbf{Proof of Theorem~\ref{thm:scale-equivariance}.}
The marginal medians are unchanged by multiplying a factor distance by a positive constant.  The empirical marginal scales transform as
\[
 \widehat s_M'=c_M\widehat s_M,
 \qquad
 \widehat s_N'=c_N\widehat s_N.
\]
Therefore
\[
 \frac{d_M'(p,X)^2}{(\widehat s_M')^2}
 +
 \frac{d_N'(q,Y)^2}{(\widehat s_N')^2}
 =
 \frac{d_M(p,X)^2}{\widehat s_M^2}
 +
 \frac{d_N(q,Y)^2}{\widehat s_N^2}.
\]
The standardized empirical objective is exactly unchanged, hence so is its minimizer.

\medskip 
\noindent \textbf{Proof of Theorem~\ref{thm:calibration-consistency}.}
The convergence of $\widehat\alpha_{\rm sc}$ follows from the continuous mapping theorem applied to \eqref{eq:alpha-sc}.  For the median,
\[
 d_1(\widehat m_{n,\widehat\alpha_{\rm sc}},m_{\alpha_{\rm sc}})
 \le
 d_1(\widehat m_{n,\widehat\alpha_{\rm sc}},m_{\widehat\alpha_{\rm sc}})
 +
 d_1(m_{\widehat\alpha_{\rm sc}},m_{\alpha_{\rm sc}}).
\]
The first term is $o_{\Pbb}(1)$ by uniform consistency; the second is $o_{\Pbb}(1)$ by continuity of the path.

\medskip
\noindent \textbf{Proof of Theorem~\ref{thm:calibrated-clt}.}
The empirical estimating equation is
\[
 G_{n,\widehat\alpha_{\rm sc}}(\widehat m_{n,\widehat\alpha_{\rm sc}})=0.
\]
Let $v_n=\log_{m_0}(\widehat m_{n,\widehat\alpha_{\rm sc}})$.  A first-order expansion in $(m,\alpha)$ around $(m_0,\alpha_{\rm sc})$ gives
\[
 0=
 G_{n,\alpha_{\rm sc}}(m_0)
 +A_0v_n
 +B_0(\widehat\alpha_{\rm sc}-\alpha_{\rm sc})
 +r_n,
\]
with $\sqrt n\|r_n\|=o_{\Pbb}(1)$ by the uniform differentiability and empirical-process assumptions.  Since
\[
 \sqrt nG_{n,\alpha_{\rm sc}}(m_0)
 =\frac1{\sqrt n}\sum_{i=1}^n\psi_0(Z_i),
\]
and \eqref{eq:alpha-if} holds, solving for $\sqrt n v_n$ yields \eqref{eq:calibrated-linearization}.  The central limit theorem gives \eqref{eq:calibrated-clt}.

\medskip 
\noindent \textbf{Proof of Lemma~\ref{lem:h-monotone}.}
Differentiate the quotient.  The numerator derivative is $A+B$, and the denominator derivative is $A-B$.  Hence
\[
 \partial_\alpha h_\alpha
 =
 \frac{(A+B)(\alpha A+(2-\alpha)B)-(\alpha A-(2-\alpha)B)(A-B)}{\{\alpha A+(2-\alpha)B\}^2}
 =
 \frac{4AB}{\{\alpha A+(2-\alpha)B\}^2}.
\]

\medskip  
\noindent \textbf{Proof of Theorem~\ref{thm:unique-balance-fixed-m}.}
Continuity follows from dominated convergence since $|h_\alpha|\le1$ and $h_\alpha$ is continuous in $\alpha$ almost surely.  Strict monotonicity follows from Lemma \ref{lem:h-monotone} and the assumption $P(AB>0)>0$.  Because $B>0$ almost surely,
\[
 \lim_{\alpha\downarrow0}h_\alpha(Z;m)=-1
\]
almost surely.  Because $A>0$ almost surely,
\[
 \lim_{\alpha\uparrow2}h_\alpha(Z;m)=1
\]
almost surely.  Dominated convergence gives the endpoint limits for $H_\alpha(m)$.  The intermediate value theorem and strict monotonicity give existence and uniqueness.

\medskip 
\noindent \textbf{Proof of Theorem~\ref{thm:balanced-clt}.}
Work in a normal coordinate chart at $m_0$ and write
\[
 \delta_n=
 \begin{pmatrix}
 \log_{m_0}\widehat m\\
 \widehat\alpha-\alpha_0
 \end{pmatrix}.
\]
A Taylor expansion of the population estimating map and stochastic equicontinuity of the empirical process give
\[
 0=P_n\Xi_{\widehat\alpha}(\cdot;\widehat m)+o_{\Pbb}(n^{-1/2})
 =P_n\Xi_{\alpha_0}(\cdot;m_0)+J\delta_n+r_n,
\]
where $\sqrt n\|r_n\|=o_{\Pbb}(1)$.  Therefore
\[
 \sqrt n\delta_n
 =-J^{-1}\frac1{\sqrt n}\sum_{i=1}^n\Xi_{\alpha_0}(Z_i;m_0)+o_{\Pbb}(1).
\]
The multivariate CLT yields \eqref{eq:balanced-normal}.



\bibliographystyle{dcu}
\bibliography{references}


\end{document}